\documentclass[12pt]{amsart}
\usepackage[utf8]{inputenc}
\usepackage[OT4]{fontenc}
\usepackage{amsmath}
\usepackage{amsfonts}
\usepackage{amssymb}
\usepackage{amsthm}
\usepackage{bbm}
\usepackage{hyperref}
\renewcommand{\sqrt}[1]{\left( #1 \right)^\frac12}

\newcommand{\e}[0]{\mathrm{e}}
\newcommand{\E}[0]{\mathbb{E}}
\newcommand{\T}[0]{\mathbb{T}}
\newcommand{\D}[0]{\mathbb{D}}

\renewcommand{\d}[0]{\mathrm{d}}

\newcommand{\indnorm}[1]{\left\|#1\right\|_{\mathrm{ind}}}
\newcommand{\Lnorm}[1]{\left\|#1\right\|_{L^1\left(\ell^2\right)}}

\newcommand{\p}[1]{\left(#1\right)}

\newcommand{\beq}{\begin{equation}}
\newcommand{\eeq}{\end{equation}}
\newcommand{\supp}[0]{\mathrm{supp}\ }

\author{Maciej Rzeszut}

\address{Jan Kochanowski University in Kielce}
\email{maciej.rzeszut@gmail.com}

\title{New examples of Fourier multipliers on $H^1\p{\mathbb{D}^2}$ revisited }

\newtheorem{thm}{Theorem}
\newtheorem{lem}[thm]{Lemma}
\newtheorem{deff}[thm]{Definition}
\newtheorem{cor}[thm]{Corollary}

\numberwithin{thm}{section}

\allowdisplaybreaks

\begin{document}

\begin{abstract}
We show yet another family of examples of idempotent Fourier multipliers on $H^1\p{\mathbb{D}^2}$. The proof differs from the old result \cite{RzW} and gets rid of arithmetical assumptions.
\end{abstract}

\maketitle
\section{Introduction}
We will identify the elements of the Hardy space $H^1\p{\D^2}$ with their limits on $\T^2$, i.e. with elements of the space
\beq \overline{\mathrm{span}}\left\{\e^{2\pi i \langle n,t\rangle}:n_1,n_2\geq 0\right\}.\eeq
On the one-variable Hardy space $H^1\p{\D}$, the problem of classifying bounded operators $T$ such that 
\beq \label{eq:ble}\widehat{Tf}=\mathbbm{1}_A \widehat{f}\eeq
for some set $A$ (called idempotent Fourier multipliers) has been solved completely \cite{K}. The analogous question for $H^1\p{\D^2}$ remains open. \par
In \cite{RzW}, we proposed the following method of constructing idempotent Fourier multipliers on $H^1\p{\D^2}$. Take sequences $d_k$ and $N_k$ of natural numbers such that $\frac{d_k}{N_{k+1}}$ is bounded, $\frac{d_k}{N_k}\to\infty$, $N_k<N_{k+1}$ and $N_k\mid N_{k+1}$. Then the set
\beq \bigcup_k \left\{\p{n_1,n_2}: n_1+n_2=d_k,\ N_k\mid n_1\right\}\eeq
can be taken as $A$ in \eqref{eq:ble}. Here, we are going to present a different proof that does not need the divisibility assumption. \par
The reduction of a two-parameter scalar-valued inequality to a one-parameter vector-valued one was done in \cite{RzW} as follows. First, by means of tensoring a Payley projection associated with the lacunary sequence $d_k$ with identity, we can reduce the problem to functions consisting of characters of the form $\p{n_1,n_2}$, where $n_1+n_2=d_k$. Then, our multiplier acts on the $k$-th generation of functions as a conditional expectation, which reduces the problem to Theorem 2.11
\section{Main result}
\begin{deff}For a sequence $\left(f_k:k\geq 1\right)$ of functions in $L^1\left(\Omega,\mathcal{F},\mu,\right)$ we define a norm \beq \indnorm{\left(f_k:k\geq 1\right)}=\int_{\Omega^\infty} \sqrt{\sum_{k=1}^\infty\left|f_k\left(\omega_k\right)\right|^2}\d \mu^{\otimes \infty}( \omega)\eeq and the space $\left(\bigoplus_{k=1}^\infty L^1\right)_{\mathrm{ind}}$ of such seqences for which this norm is finite.\end{deff}
\begin{deff}A dyadic atom is a function $a\in L^2[0,1]$ of mean $0$, supported on a dyadic interval $I$ such that $\|a\|_{L^2}\leq |I|^{-\frac12}$. \end{deff}

\begin{thm}[\cite{coif}]\label{atdec}If $f\in H^1(\delta)$, then there exists a sequence of atoms $\left(a_k:k\geq 1\right)$ and scalars $\left(c_k:k\geq 1\right)$ such that \beq f-\E f= \sum_{k=1}^\infty c_k a_k,\quad \sum_{k=1}^\infty \left|c_k\right|\lesssim \|f\|_{H^1(\delta)}.\eeq \end{thm}
\begin{cor}\label{atext}A bounded linear (sublinear) operator $T:H^1_0(\delta)\to X$, where $X$ is a Banach space (a Banach lattice), such that $\|Ta\|_X\leq C$ for any atom $a$, satisfies $\left\|T:H^1(\delta)\to X\right\|\lesssim C$. \end{cor}
\begin{proof}Let $f=\sum_{k=1}^\infty c_k a_k$ be the decomposition given by Theorem \ref{atdec}. By continuity of $T$, $Tf=\sum_{k=1}^\infty c_k T a_k$ (or $\left|Tf\right|\leq \sum_{k=1}^\infty \left|c_k T a_k\right|$). Thus $\|Tf\|_X\leq \sum_{k=1}^\infty \left|c_k\right|\cdot \left\|Ta_k\right\|_X\lesssim C\|f\|_{H^1(\delta)}$. \end{proof} Care has to be taken, as this definition of an atom (precisely an (1,2)-atom) differs from the more widely used $(1,\infty)$-atoms satisfying $\|a\|_{L^\infty}\leq |I|^{-1}$. Corollary \ref{atext} is nontrivial if we drop the a priori boundedness of $T$ and false if we additionally replace $(1,2)$-atoms with $(1,\infty)$ ones (see \cite{bownik},\cite{msv}). 
\begin{thm}\label{dyadicrbdd}Let $\left(\mathcal{F}_n:n\geq 0\right)$ be the dyadic filtration on $[0,1]$, $\left(m_k:k\geq 1\right)$ be an increasing sequence of integers and $s\geq 0$ be an integer. Suppose we are given a sequence of sublinear operators $T_k$ acting on $\mathcal{F}_{m_k}$-measurabe functions such that \beq \left\|T_k:L^1\left([0,1],\mathcal{F}_{m_k}\right)\to L^1[0,1]\right\|\leq C_1\eeq and \beq \label{eq:tkl2}\left\|T_k f\right\|_{L^2}\leq C_2 |I|^{\frac{1}{2}}\left\|f\right\|_{L^2} \eeq whenever the function $f$ is supported on a dyadic interval $I$ of length $2^{-m}$ and there exists $j$ such that $m\leq m_{j}$ and $k\geq j+s$. Then for any sequence of $\mathcal{F}_{m_k}$-measurable functions $f_k$ we have \beq \label{eq:rbddmain} \indnorm{\left(T_kf_k:k\geq 1\right)}\lesssim \left(C_1 s^{\frac12}+ C_2\right) \Lnorm{\left(f_k:k\geq 1\right)}.\eeq \end{thm}
\begin{proof} Let \beq f=\sum_{k=1}^\infty r_{m_k +1}f_k.\eeq Then \beq \label{eq:deltai}\Delta_k f=\left\{\begin{array}{ll}r_{m_j+1}f_j&\text{if }k=m_{j}+1\\ 0&\text{otherwise}\end{array}\right.\eeq and consequently \beq \left\|f\right\|_{H^1\left(\delta\right)}= \Lnorm{\left(f_k\right)}.\eeq Therefore it suffices to prove that \beq \indnorm{\left(T_k f_k: k\geq 1\right)}\lesssim C_1s^\frac12+C_2 \eeq when $f$ is an atom. Indeed, if this is true, then the operators \beq H^1(\delta)\ni f\mapsto \left(T_k f_k:1\leq k\leq K\right)\in\left(\bigoplus L^1[0,1]\right)_{\mathrm{ind}}\eeq are a priori bounded, because \begin{eqnarray} \Lnorm{\left(T_kf_k:1\leq k\leq K\right)}&\leq&\left\| \left(T_kf_k:1\leq k\leq K\right)\right\|_{L^1\left(\ell^1\right)}\\ &\leq& C_1 \left\| \left(f_k:1\leq k\leq K\right)\right\|_{L^1\left(\ell^1\right)}\\ &\leq& C_1 K^\frac12\left\| \left(f_k:1\leq k\leq K\right)\right\|_{L^1\left(\ell^2\right)}\\ &\leq& C_1 K^\frac12 \|f\|_{H^1(\delta)}\end{eqnarray} and by Corollary \ref{atext} their norms are $\lesssim C_1s^\frac12+C_2$, yielding \eqref{eq:rbddmain} as $K\to\infty$. \par
Suppose now that $f$ is an atom supported on a dyadic interval $I$, where $|I|=2^{-m}$. By \eqref{eq:deltai}, \beq f_k=r_{m_k+1}\Delta_{m_k+1}f.\eeq Let \beq j=\min\left\{i: m_i\geq m\right\}.\eeq Then for $k<j$, we have $m_k+1\leq m$, thus $\Delta_{m_k+1}f=\Delta_{m_k+1}\E_m f=0$. If $k\geq j$, then $m_k\geq m$, thus $\E_{m_k}f,\E_{m_k+1}f$ are supported on $I$, and so is $f_k=r_{m_k+1}\left(\E_{m_k+1}f-\E_{m_k}f\right)$. Therefore \begin{eqnarray} \indnorm{\left(T_k f_k:j+s>k\geq j\right)} &\leq & \left\|\left(T_k f_k:j+s>k\geq j\right)\right\|_{L^1\left(\ell^1\right)}\\ &=& \sum_{j\leq k<j+s}\left\|T_k f_k\right\|_{L^1}\\&\leq & C_1\sum_{j\leq k<j+s}\left\|f_k\right\|_{L^1}\\ &\leq & C_1 |I|^{\frac{1}{2}}\sum_{j\leq k<j+s}\left\|f_k\right\|_{L^2}\\ &\leq& C_1|I|^\frac12 s^\frac12 \sqrt{\sum_{j\leq k<j+s}\left\|f_k\right\|_{L^2}^2}\end{eqnarray}
and \begin{eqnarray} \indnorm{\left(T_k f_k:k\geq j+s\right)}&\leq& \left\|\left(T_k f_k:k\geq j+s\right)\right\|_{L^2\left(\ell^2\right)}\\ &=& \left(\sum_{k\geq j+s}\left\|T_k f_k\right\|_{L^2}^2\right)^\frac12\\ &\leq & C_2|I|^\frac12 \left(\sum_{k\geq j+s}\left\|f_k\right\|_{L^2}^2\right)^\frac12.\end{eqnarray}
Ultimately
\begin{eqnarray} \indnorm{\left(T_k f_k:k\geq 1\right)}&=& \indnorm{\left(T_k f_k:k\geq j\right)}\\ &\leq & \indnorm{\left(T_k f_k:j+s>k\geq j\right)}+ \indnorm{\left(T_k f_k:k\geq j+s\right)}\\ &\leq& C_1 |I|^\frac12 s^\frac12 \sqrt{\sum_{j\leq k<j+s}\left\|f_k\right\|_{L^2}^2}+C_2|I|^\frac12\sqrt{\sum_{k\geq j+s}\left\|f_k\right\|_{L^2}^2}\\ &\lesssim& \left(C_1 s^\frac12+C_2\right)|I|^\frac12 \sqrt{\sum_{k\geq j}\left\|f_k\right\|_{L^2}^2} \\ &=& \left(C_1 s^\frac12+C_2\right)|I|^\frac12\sqrt{\sum_{k\geq j}\left\|\Delta_{m_k+1}f\right\|_{L^2}^2}\\ &\leq & \left(C_1 s^\frac12+C_2\right)|I|^\frac12 \|f\|_{L^2}\\ &\leq & C_1 s^\frac12+C_2\end{eqnarray} as desired. \end{proof}
We identify $[0,1]$ with $\T$ and for $f\in L^1(\T)$ denote \beq \tau_{x_0}f(x)= f\left(x-x_0\right), \quad \E^*_N = \frac1N \sum_{j=0}^{N-1} \tau_{\frac{j}{N}},\eeq\beq \E_N f= \sum_{k=0}^{N-1} N\mathbbm{1}_{\left[\frac{k}{N},\frac{k+1}{N}\right)}\int_{\frac{k}{N}}^\frac{k+1}{N} f(x)\d x.\eeq
\begin{lem}\label{enl2}If $f\in L^2(\T)$ is supported on an interval $I$, then \beq \left\|\E^*_N f\right\|_{L^2}\leq \left(|I|+\frac{2}{N}\right)^\frac12\|f\|_{L^2}.\eeq\end{lem}
\begin{proof}Let $J_k= \left[\frac{k}{N},\frac{k+1}{N}\right)$. Suppose first that $\supp f\subset J_{k_0}$. Then for different values of $k$, the functions $\tau_{\frac{k}{N}} f$ are supported on disjoint intervals $J_{k+k_0}$. Thus \begin{eqnarray}\left\|\E^*_N f\right\|_{L^2}^2&=& \int_0^1 \left|\E_N^* f(x)\right|^2 \d x\\ &=& \sum_{k=0}^{N-1} \int_{J_{k+k_0}}\left|\frac{1}{N}\tau_{\frac{k}{N}}f(x)\right|^2 \d x\\ &=& \frac{1}{N^2} \sum_{k=0}^{N-1} \int_{J_{k_0}}\left|f(x)\right|^2\d x\\ &=&\label{eq:suppJ0}\frac{1}{N}\|f\|_{L^2}^2. \end{eqnarray} Now let $f$ be supported on $I$. For at most 2 elements of the set $\left\{k: I\cap J_k\neq \emptyset\right\}$ we have $I\cap J_k\neq J_k$, so \beq \label{eq:minus2}\frac{1}{N}\left(\left|\left\{k: I\cap J_k\neq \emptyset\right\}\right|-2\right)\leq |I|.\eeq Denoting $f_k=f\mathbbm{1}_{J_k}$ and utilising \eqref{eq:suppJ0} and \eqref{eq:minus2}, we get \begin{eqnarray} \left\|\E^*_Nf\right\|_{L^2}&\leq & \sum_{k=0}^{N-1} \left\|\E^*_Nf_k\right\|_{L^2}\\  &\leq &\frac{1}{N^\frac12} \sum_{k=0}^{N-1} \left\|f_k\right\|_{L^2}\\ &= &\frac{1}{N^\frac12} \sum_{I\cap J_k\neq \emptyset} \left\|f_k\right\|_{L^2}\\ &\leq& \sqrt{\frac{\left|\left\{k: I\cap J_k\neq \emptyset\right\}\right|}{N}} \sqrt{\sum_{I\cap J_k\neq \emptyset} \left\|f_k\right\|^2_{L^2}}\\ &\leq& \sqrt{|I|+\frac2N}\sqrt{\sum_{k=0}^{N-1}\int_{J_k} \left|f_k(x)\right|^2\d x}\\ &=& \sqrt{|I|+\frac2N}\|f\|_{L^2}\end{eqnarray}as desired. \end{proof}

\begin{lem}\label{schodkapr}Let $f$ be absolutely continuous on $\T$. Then \beq\left|\left(\mathrm{id}-\E_N\right)f\right|\leq \frac{1}{N}\E_N \left|f'\right|\eeq almost everywhere. \end{lem}
\begin{proof} Fix $x$ and let $I$ be the interval of the form $\left[\frac{k}{N},\frac{k+1}{N}\right)$ containing $x$. Then \begin{eqnarray}\left|\left(\mathrm{id}-\E_N\right)f(x)\right|&=&\left|N\int_{I}\left(f(x)-f(y)\right)\d y\right|\\ &=& \left|N\int_{I}\int_{y}^x f'(z)\d z\d y\right|\\ &\leq&  N\int_{I^2}\left|f'(z)\right| \mathbbm{1}_{\mathrm{conv}\{x,y\}}(z)\d z \d y\\ &\leq& \int_{I}\left|f'(z)\right|\d z\\ &=&\frac{1}{N} \E_N\left|f'\right|(x).\end{eqnarray}
\end{proof}
Let us recall the Stein multiplier theorem \cite{steinmult}. 
\begin{thm}\label{steinmult}Let $\left(\mu(n): n\geq 0\right)$ be a sequence of scalars such that \beq \left|\mu(n)\right|\leq C,\quad \left(n+1\right)\left|\mu(n+1)-\mu(n)\right|\leq C.\eeq Then $\left(\mu_n: n\geq 0\right)$ defines a bounded Fourier multiplier on $H^1(\D)$ of norm $\lesssim C$. \end{thm}
From now on $\left(d_k:k\geq 1\right)$ will be a fixed lacunary sequence of integers, i.e. \beq d_{k+1}\geq (1+\alpha)d_k\eeq for some $\alpha>0$. We would like to extend exhibit a version of Theorem \ref{dyadicrbdd} for trigonometric polynomials and operators $\E^*_N$. Denote the spaces of analytic polynomials by \beq H^1_n(\D)= \mathrm{span}\left\{\e^{2\pi ijt}:0\leq j\leq n\right\}\subset H^1(\D) \eeq and the Fej\'er kernel by \beq K_n(t)=\sum_{j=-n}^n \left(1-\frac{|j|}{n}\right)\e^{2\pi i j t}.\eeq

\begin{lem}\label{fejerrbdd}Let $f_k \in H^1_{d_k}(\D)$. Then \beq \Lnorm{\left(\frac{f_k'}{d_k}\right)_{k=1}^\infty} = 2\pi \Lnorm{\left(f_k\ast \left(K_{d_k}\cdot \e^{2\pi i d_{k} t}\right)\right)_{k=1}^n}\leq C_\alpha \Lnorm{\left(f_k\right)_{k=1}^\infty}\eeq for some $C_\alpha\lesssim 1+\alpha^{-3}$. \end{lem}
\begin{proof}Let us denote \beq D_k=\sum_{j=1}^k d_k.\eeq By lacunarity of $d_k$,\begin{eqnarray} D_k&=& d_k\sum_{j=1}^k \frac{d_j}{d_k}\\ &\leq& d_k\sum_{j=1}^k\frac{1}{\left(1+\alpha\right)^{k-j}}\\ &\leq & \frac{d_k}{1-\frac{1}{1+\alpha}}\\ &=&\label{lacun} \left(1+\frac{1}{\alpha}\right)d_k.\end{eqnarray} For any sequence of signs $\varepsilon_k\in\{-1,1\}$ we define, following \cite{wojtsurv}, a sequence $\mu_\varepsilon(n)$ such that \beq \mu_\varepsilon\left(3D_{k-1}+j d_k\right) = \left\{\begin{array}{lll} \varepsilon_k & \text{for} &j\in\{1,2\}\\ 0&\text{for}& j=0\end{array}\right.\eeq and $\mu_\varepsilon$ is affine on each interval of the form $\left[3D_{k-1}+jd_k, 3D_{k-1}+(j+1)d_k\right]$, where $j\in\{0,1,2\}$. For any $n\in \left(3D_{k-1},3D_k\right]$ we have \beq \label{eq:ordalpha}n\left|\mu_\varepsilon(n)-\mu_\varepsilon(n-1)\right|\leq 3D_k\cdot \frac{1}{d_k}\leq 3\left(1+\alpha^{-1}\right)\eeq by \eqref{lacun}, so $\mu_\varepsilon$ satisfies the assumptions of Theorem \ref{steinmult}. Let $S_\varepsilon$ denote the associated operator on $H^1(\D)$. Let $g_k$ satisfy $\supp \widehat{g_k}\subset\left[3D_{k-1}+d_k,3D_{k-1}+2d_k\right]$. By definition of $\mu_\varepsilon$, \beq S_\varepsilon g_k= \varepsilon_k g_k.\eeq Therefore \beq \left\|\sum_{k=1}^n \varepsilon_k g_k\right\|_{L^1}= \left\|S_\varepsilon\sum_{k=1}^n g_k\right\|_{L^1}\lesssim_\alpha  \left\|\sum_{k=1}^n g_k\right\|_{L^1}.\label{eq:dupa}\eeq
Applying \eqref{eq:dupa} to $\varepsilon_k g_k$ in place of $g_k$ we get the reverse estimate, so \beq \left\|\sum_{k=1}^n g_k\right\|_{L^1}\simeq_\alpha \left\|\sum_{k=1}^n \varepsilon_k g_k\right\|_{L^1}.\label{eq:muewe}\eeq Applying the Khintchine inequality to the right hand side of \eqref{eq:muewe}, we get \beq \left\|\sum_{k=1}^n g_k\right\|_{L^1}\simeq_\alpha \left\|\sqrt{\sum_{k=1}^n \left|g_k\right|^2}\right\|_{L^1}.\label{eq:sqfchar}\eeq Now let us notice that Theorem \ref{steinmult}, by an argument identical to the $S_\varepsilon$ case, implies the boundedness of an operator $K$ on $H^1(\D)$ given by the multiplier $\widehat{K}$ satisfying \beq \widehat{K}\left(3D_{k-1}+jd_k\right)= \left\{\begin{array}{lll} 0 &\text{for}& j\in\{0,1\}\\ 1&\text{for}& j=2\end{array}\right.\eeq and $\widehat{K}$ affine on each interval of the form $\left[3D_{k-1}+jd_k, 3D_{k-1}+(j+1)d_k\right]$, where $j\in\{0,1,2\}$. Also, for $g_k$ such that $\widehat{g_k}\subset \left[3D_{k-1}+d_k,3D_k\right]$ we have \beq \label{eq:bezminusa}Kg_k = g_k\ast \left(\e^{2\pi i \left(3D_{k-1}+2d_k\right)t} K_{d_k}\right).\eeq Therefore \begin{eqnarray}\Lnorm{\left(\frac{f_k'}{d_k}\right)_{k=1}^n}&=&\int_\T \sqrt{\sum_{k=1}^n \left|\frac{f_k'(t)}{d_k}\right|^2}\d t\\ &=& \int_\T \sqrt{\sum_{k=1}^n \left|\frac{2\pi i}{d_k}\sum_{0\leq j<d_k}j\widehat{f_k}(j)\e^{2\pi i j t}\right|^2}\d t\\ &=& 2\pi \int_\T \sqrt{\sum_{k=1}^n \left|\sum_{0\leq j<d_k}\widehat{f_k}(j)\widehat{K_{d_k}}\left(j-d_k\right)\e^{2\pi i j t}\right|^2}\d t\\ &=& 2\pi \Lnorm{\left(f_k\ast \left(K_{d_k}\cdot \e^{2\pi i d_{k} t}\right)\right)_{k=1}^n} \\ &\simeq& \Lnorm{\left(\e^{2\pi i \left(D_k+d_k\right)t}\left(f_k\ast \left(K_{d_k}\cdot \e^{2\pi i d_{k} t}\right)\right)\right)_{k=1}^n}\\ &=& \Lnorm{\left(\left(f_k\cdot\e^{2\pi i \left(D_k+d_k\right)t}\right) \ast\left(K_{d_k}\cdot \e^{2\pi i \left(D_k+2d_k\right)t}\right)\right)_{k=1}^n}\\ \text{(By \eqref{eq:sqfchar})}&\simeq_\alpha& \left\|\sum_{k=1}^n\left(f_k\cdot\e^{2\pi i \left(D_k+d_k\right)t}\right) \ast\left(K_{d_k}\cdot \e^{2\pi i \left(D_k+2d_k\right)t}\right)\right\|_{L^1}\\ \text{(By \eqref{eq:bezminusa})}&=& \left\|\sum_{k=1}^n K\left(f_k\cdot\e^{2\pi i \left(D_k+d_k\right)t}\right)\right\|_{L^1}\\ \text{($K$ is bounded)} &\lesssim_\alpha& \left\|\sum_{k=1}^n f_k\cdot\e^{2\pi i \left(D_k+d_k\right)t}\right\|_{L^1}\\ \text{(By \eqref{eq:sqfchar})} &\simeq_\alpha& \Lnorm{\left(f_k\right)_{k=1}^n}.\end{eqnarray} The bound $C_\alpha\lesssim 1+\alpha^{-3}$, which is probably not optimal, comes from the fact that each of the three $\lesssim_\alpha$ above originated form a single application of one of the operators $K,S_\varepsilon$, which are of norm $\lesssim 1+\alpha^{-1}$ by \eqref{eq:ordalpha} and Theorem \ref{steinmult}. 
\end{proof}

\begin{lem}\label{discretization}Let $f_k\in H^1_{d_k}(\D)$ and $\left(M_k:k\geq 1\right)$ be a sequence of integers such that $d_k\leq \varepsilon C_\alpha^{-1} M_k$, where $\varepsilon<\frac{1}{2}$. Then \beq \label{eq:discretization}\left(1-\varepsilon\right)\Lnorm{\left(\E_{M_k}\left|f_k\right|\right)_{k=1}^\infty}\leq \Lnorm{\left(f_k\right)_{k=1}^\infty}\leq \frac{1-\varepsilon}{1-2\varepsilon} \Lnorm{\left(\E_{M_k}f_k\right)_{k=1}^\infty}.\eeq \end{lem}
\begin{proof}Clearly it is enough to prove \eqref{eq:discretization} in the case when only finitely many, say $f_1,\ldots,f_n$, of given functions are nonzero. In order to save space, we will omit indices and write $\Lnorm{\left(\cdot\right)}$ instead of $\Lnorm{\left(\cdot\right)_{k=1}^n}$. We will prove by induction with respect to $r$ that \beq \label{eq:indstep}\Lnorm{\left(\E_{M_k}\left|\frac{f_k'}{M_k}\right|\right)}\leq \sum_{j=1}^{r-1}  \Lnorm{\left(\frac{f_k'^{(j)}}{M_k^j}\right)}+ \Lnorm{\left(\E_{M_k}\left|\frac{f_k'^{(r)}}{M_k^{r}}\right|\right)}.\eeq For $r=1$ there is nothing to prove. Let us assume \eqref{eq:indstep} for some $r$. By the pointwise estimates in Lemma \ref{schodkapr} and $\left||f|'\right|\leq \left|f'\right|$, 
\begin{eqnarray} \Lnorm{\left(\E_{M_k}\left|\frac{f_k'^{(r)}}{M_k^{r}}\right|\right)} &\leq & \Lnorm{\left(\frac{f_k'^{(r)}}{M_k^{r}}\right)}+ \Lnorm{\left(\left(\mathrm{id}-\E_{M_k}\right)\left|\frac{f_k'^{(r)}}{M_k^{r}}\right|\right)}\\ &\leq& \Lnorm{\left(\frac{f_k'^{(r)}}{M_k^{r}}\right)}+ \Lnorm{\left(\frac{1}{M_k}\E_{M_k} \left|\left|\frac{f_k'^{(r)}}{M_k^{r}}\right|'\right|\right)}\\ &\leq & \Lnorm{\left(\frac{f_k'^{(r)}}{M_k^{r}}\right)}+ \Lnorm{\left(\E_{M_k} \left|\frac{f_k'^{(r+1)}}{M_k^{r+1}}\right|\right)}, \end{eqnarray} which plugged into \eqref{eq:indstep} gives exactly the same with $r+1$ in place of $r$. Therefore
\begin{eqnarray} \Lnorm{\left(\E_{M_k}\left|\frac{f_k'}{M_k}\right|\right)}
&\leq& \sum_{j=1}^{r-1}  \Lnorm{\left(\frac{f_k'^{(j)}}{M_k^j}\right)}+ \Lnorm{\left(\E_{M_k}\left|\frac{f_k'^{(r)}}{M_k^{r}}\right|\right)}\\ 
\text{(Cauchy-Schwarz)} &\leq & \sum_{j=1}^{r-1}  \Lnorm{\left(\frac{f_k'^{(j)}}{M_k^j}\right)}+ n^\frac{1}{2}\Lnorm{\left(\frac{f_k'^{(r)}}{M_k^{r}}\right)}\\
&\leq& \sum_{j=1}^{r-1} \left(\varepsilon C_\alpha^{-1}\right)^j \Lnorm{\left(\frac{f_k'^{(j)}}{d_k^j}\right)}+ n^\frac{1}{2}\left(\varepsilon C_\alpha^{-1}\right)^r\Lnorm{\left(\frac{f_k'^{(r)}}{d_k^{r}}\right)}\\
\text{(Iterating Lemma \ref{fejerrbdd})} & \leq & \sum_{j=1}^{r-1} \varepsilon^j \Lnorm{\left(f_k\right)}+ n^\frac{1}{2}\varepsilon^r\Lnorm{\left(f_k\right)}\\ &\stackrel{r\to\infty}{\longrightarrow}& \frac{\varepsilon}{1-\varepsilon}\Lnorm{\left(f_k\right)}.
\end{eqnarray}
Combining this with another usage of Lemma \ref{schodkapr} again we get \begin{eqnarray} \Lnorm{\left(\E_{M_k}\left|f_k\right|\right)} &\leq & \Lnorm{\left(f_k\right)}+ \Lnorm{\left(\left(\mathrm{id}-\E_{M_k}\right)\left|f_k\right|\right)}\\ 
&\leq & \Lnorm{\left(f_k\right)}+ \Lnorm{\left(\E_{M_k}\left|\frac{\left|f_k\right|'}{M_k}\right|\right)}\\ &\leq& \Lnorm{\left(f_k\right)}+ \frac{\varepsilon}{1-\varepsilon}\Lnorm{\left(f_k\right)}\\ &=& \frac{1}{1-\varepsilon}\Lnorm{\left(f_k\right)},\end{eqnarray} which proves the left hand side of \eqref{eq:discretization}. Similarly
\begin{eqnarray} \Lnorm{\left(f_k\right)}&\leq & \Lnorm{\left(\E_{M_k}f_k\right)}+ \Lnorm{\left(\left(\mathrm{id}-\E_{M_k}\right)f_k\right)}\\ &\leq & \Lnorm{\left(\E_{M_k}f_k\right)}+ \Lnorm{\left(\E_{M_k}\left|\frac{f_k'}{M_k}\right|\right)}\\ &\leq & \Lnorm{\left(\E_{M_k}f_k\right)}+ \frac{\varepsilon}{1-\varepsilon}\Lnorm{\left(f_k\right)},\end{eqnarray} proving the other inequality. 
\end{proof}
\begin{thm}\label{oldrev}Let $\left(d_k:k\geq 1\right)$, $\left(N_k:k\geq 1\right)$ be sequences of integers such that $d_{k+1}\geq (1+\alpha) d_k$ for some $\alpha>0$ and $d_k\leq \beta N_{k+s}$ for some $\beta>0$ and an integer $s\geq 0$. Then for any $f_k\in H^1_{d_k}(\D)$ the inequality 
\beq \Lnorm{\left(f_k\right)_{k=1}^\infty}\gtrsim \indnorm{\left(\E^*_{N_k}\left|f_k\right|\right)_{k=1}^\infty}\eeq holds with a constatnt dependent only on $s,\alpha, \sup\frac{d_k}{N_{k+s}}$. \end{thm}
\begin{proof}Without loss of generality we may assume $\alpha>1$. Indeed, assume the weakened version and let $q>\frac{1}{\alpha}$ be an integer. Then $\left(1+\alpha\right)^q>2$ and thus $d_{k+q}> 2d_k$. For any $r\in\{1,\ldots,q\}$, the sequences $\left(d_{r+kq}:k\geq 0\right)$ and $\left(N_{r+kq}:k\geq 0\right)$ satisfy the assumptions of Theorem \ref{oldrev} with $2$ in place of $\alpha$ and $\beta^q$ in place of $\beta$. Thus \begin{eqnarray} \Lnorm{\left(f_k\right)_{k=1}^\infty}&\geq & q^{-\frac{1}{2}}\sum_{r=1}^q \Lnorm{\left(f_{r+kq}\right)_{k=1}^\infty}\\ &\gtrsim & q^{-\frac{1}{2}}\sum_{r=1}^q \indnorm{\left(\E^*_{N_{r+kq}}\left|f_{r+kq}\right|\right)_{k=1}^\infty}\\ &\geq & q^{-\frac{1}{2}}\indnorm{\left(\E^*_{N_k}\left|f_k\right|\right)_{k=1}^\infty}.\end{eqnarray}
Since now $d_{k+1}>2d_k$, the sequence $m_k=\left\lceil \log_2\left(3C_\alpha d_k\right)\right\rceil$ is increasing.  Also, \beq m_k-1< \log_2\left(3C_\alpha d_k\right)\leq m_k\eeq guarantees $2^{m_k}\simeq d_k$. Let us define operators $T_k$ acting on $L^1[0,1]$ by \beq T_kf= \E^*_{N_k}\left|f\right|.\eeq Take $f$ supported on a dyadic interval $I$ of length $2^{-m}$, an integer $j$ such that $m\leq m_j$ and $k$ such that $k\geq j+s$. Then \beq N_k\gtrsim d_{k-s}\gtrsim 2^{m_{k-s}}\geq 2^{m_j}\geq 2^m= |I|^{-1},\eeq so for any $f\in L^2$, by Lemma \ref{enl2} applied to $|f|$,  \begin{eqnarray} \left\|T_k f\right\|_{L^2}&=& \left\|\E^*_{N_k}\left|f\right|\right\|_{L^2}\\ &\leq& \sqrt{|I|+\frac{2}{N_k}}\|f\|_{L^2}\\ &\lesssim& |I|^\frac{1}{2}\|f\|_{L^2},\end{eqnarray} verifying \eqref{eq:tkl2}. Thus $T_k$ together with $m_k$ satisfy the assumptions of Theorem \ref{dyadicrbdd}. Hence, for $M_k=2^{m_k}$, \begin{eqnarray} \indnorm{\left(\E^*_{N_k}\left|\left(\mathrm{id}-\E_{M_k}\right)f_k\right|\right)}&\leq& \indnorm{\left(\E^*_{N_k}\E_{M_k}\left|\frac{f_k'}{M_k}\right|\right)}\\ &\lesssim& \Lnorm{\left(\E_{M_k}\left|\frac{f_k'}{M_k}\right|\right)}\\ \text{(By Lemma \ref{discretization})}&\simeq& \Lnorm{\left(\frac{f_k'}{M_k}\right)}\\ &\simeq& \Lnorm{\left(\frac{f_k'}{d_k}\right)} \\ \label{eq:prawiekoniec}\text{(By Lemma \ref{fejerrbdd})}&\lesssim& \Lnorm{\left(f_k\right)}.\end{eqnarray} Ultimately, \begin{eqnarray} \indnorm{\left(\E^*_{N_k}\left|f_k\right|\right)}&\leq & \indnorm{\left(\E^*_{N_k}\E_{M_k}\left|f_k\right|\right)}+ \indnorm{\left(\E^*_{N_k}\left|\left(\mathrm{id}-\E_{M_k}\right)f_k\right|\right)}\\  &\lesssim& \Lnorm{\left(\E_{M_k}\left|f_k\right|\right)}+ \Lnorm{\left(f_k\right)}\\ &\lesssim& \Lnorm{\left(f_k\right)}\end{eqnarray} as desired.\end{proof}

\end{document}